\documentclass[10pt]{amsart}

\usepackage{amssymb,amsmath, graphicx, verbatim,epsfig,color,}
\usepackage[latin1]{inputenc}
\usepackage[all,cmtip]{xy}

\usepackage{array}
\usepackage{multirow}

\usepackage{varioref}
\usepackage{arydshln} %dashed lines in tables

\newcommand{\C}{\mbox{${\mathbb C}$}}

\newcommand{\p}{\mbox{${\mathbb P}$}}

\newtheorem{pr}{Proposition}[section]
\newtheorem{tm}{Theorem}[section]
\newtheorem{cor}[pr]{Corollary}

\theoremstyle{definition}		% use "definition-style" font for the rest
\newtheorem{remark}{Remark}[section]

\newcommand{\grass}{{G}(k,n)}

\newcommand{\barr}{\overline}

\newcommand{\smap}{\barr{M}}

\newcommand{\st}{_{g,r}(X,\beta)}

\newcommand{\stog}{_{0,0}(\grass,d)}
\newcommand{\stogii}{_{0,0}(\grass,2)}
\newcommand{\stogiii}{_{0,0}(\grass,3)}

 \makeatletter
\def\cdotfill{\leaders\hbox to.6em{\hss$\cdot$\hss}\hskip\z@ plus  1fill}
\makeatother

\title{Poincar\'e polynomials of stable map spaces to Grassmannians}
\author{Alberto L\'opez Mart\'in}
\thanks{The author was partially supported by the Universit\"at Z\"urich under the Forschungskredit Nr. 57104401, and by the Swiss National Science Foundation project 200020\_126756.}
\subjclass[2010]{Primary 14N35\ \ Secondary 14M15 $\cdot$ 14N15}
\address{Korea Institute for Advanced Study, Hoegiro 87, Dongdaemun-gu, Seoul 130-722, Republic of Korea }
\curraddr{Department of Mathematics, Tufts University, Bromfield-Pearson Hall, 503 Boston Avenue, Medford, MA 02155}
\email{alberto.lopez@tufts.edu}

\begin{document}
\date{\today}
\bibliographystyle{alpha}
\begin{abstract}
  In this paper, we give generating functions for the Betti numbers of $\overline{M}\stog$, the moduli stack of zero pointed genus zero degree $d$ stable maps to the Grassmannian $\grass$ for $d=2$ and $3$. \end{abstract}
\maketitle

\section{Introduction}
Moduli spaces of stable maps have been used to give answers to many problems in enumerative geometry that were inaccessible by previous methods. Let $X$ be a smooth projective variety over $\mathbb C$, and let $\beta$ be a curve class on $X$. The set of isomorphism classes of pointed maps $(C,p_1,\dots,p_r,f)$ where $C$ is a projective nonsingular curve of genus $g$, the points $p_1,\dots,p_r$ are distinct marked points on $C$, and $f$ is a morphism from $C$ to $X$ with $f_*([C])=\beta$, is denoted by $M\st$. Points in its compactification, the moduli space $\smap\st$,  parameterize maps $(C,p_1,\dots,p_r,f)$, where $C$ is a projective, connected, \emph{nodal} curve of genus $g$, the points $p_1,\dots,p_r$ are distinct \emph{nonsingular} points, and $f$ is a morphism from $C$ to $X$ such that $f_*([C])=\beta$. The stability condition is due to Kontsevich \cite{KM94} and it amounts to requiring finiteness of the automorphisms of the map. The moduli space of stable maps has been a key tool to easily formulate and predict enumerative results in the geometry of curves. We refer the reader to \cite{FP97} for further details about these moduli spaces and an introduction to Gromov-Witten theory.

The computation of Betti numbers has proven to be useful in the calculation of Chow or cohomology rings of different spaces. Getzler and Pandharipande \cite{GP06} computed the Betti numbers of degree $d$ genus zero stable maps to a projective space target using (equivariant) Serre characteristics, and their results were used later by Behrend and O'Halloran \cite{BO03} in the computation of cohomology rings of the stable map spaces $\overline{M}_{0,0} (\mathbb{P}^n , d )$. In work still in progress, Behrend computes the Betti numbers of the space of genus zero degree three stable maps to the infinite-dimensional Grassmannian  $\overline{M}_{0,0} (G(k,\C^\infty) , 3)$. In \cite{Beh}, he uses the methods developed in \cite{GP06} and introduces the modified Serre characteristic of a smooth stack.

 Our paper explores the case of a target variety other than projective space, using a method different from those mentioned above. We will consider the case of zero pointed genus zero degree $d$ stable maps to the Grassmannian variety of $k$-dimensional planes in $\C^n$ in degrees two and three. Our computations will lead to expressions for the Poincar\'e polynomials of these spaces, and hence their Betti numbers. For that, we will consider a torus action on $\smap\stog$ with isolated fixed points. In the case of torus actions on varieties, Bia{\l}ynicki-Birula \cite{BB73,BB76} showed that there exists a decomposition of the variety into cells with certain properties. It follows from his results that the cohomology of these varieties can be determined from the fixed loci of the torus action. Our computations rely on previous work of Oprea \cite{Op06} for moduli spaces of stable maps. In his paper, among other results, he obtained a Bia{\l}ynicki-Birula decomposition of the stack of stable maps and related it to previous work of Gathmann \cite[3.2.]{Op06}. Combining some of the results in \cite{Op06} (see Corollary 4, Proposition 5, Lemma 6, Corollary 14, and Proposition 15), we find that the cohomology (with $\mathbb{Q}$ coefficients) of the moduli space of stable maps considered by us is determined from the fixed loci in the manner of the usual Bia{\l}ynicki-Birula decomposition of smooth projective varieties. (The method developed here can be, in principle, applied to maps of arbitrary degree, as it relies on the combinatorial description of the fixed points of the torus action on the moduli stack. The case of pointed maps could also be treated with similar methods.)

It is relevant to mention that Str\o mme \cite{Str87} worked out the cohomology of Quot schemes parametrizing maps from the projective line to a Grassmannian variety using torus actions and the Bia{\l}ynicki-Birula decomposition mentioned above. Later, Chen generalized  the method to the case of hyperquot schemes and computed a generating function for the Poincar\'e polynomials of these, see \cite{Che01}. In \cite{Man00}, Manin calculated the virtual Poincar\'e polynomials of the moduli space of stable genus 0 maps to a generalized flag variety. 

The outline of the paper is as follows. In Section \ref{torus}, we consider certain inclusion maps between Grassmannians, and obtain lists of additional weights for the action arising from these inclusions. In Section \ref{poincare}, we calculate the Betti numbers of small degree stable map spaces to the Grassmannian. Lastly, in the Appendix we recall the definition of $q$-binomial coefficients and state some of the identities used in the proofs of the results in Section \ref{poincare}.

An application of our results appears in recent work of Chung, Hong, and Kiem \cite{CHK12}.

\section{Torus actions and Grassmannians}\label{torus}

\subsection{Torus action and fixed loci} \label{loci}

Let $\C^*=\C\setminus\{0\}$ denote the multiplicative group of complex numbers and $\mathbb{T}$ an algebraic torus. We have $\mathbb{T}\simeq(\C^*)^N$ for some $N$. Let $\{e_1,\dots,e_{N}\}$ be the standard basis of $\C^{N}$ and assume $\alpha_1,\dots,\alpha_N\in\mathbb N$ are distinct. The $(\alpha_1,\dots,\alpha_N)$-weighted action of a torus on $\C^N$ described by
$$\begin{array}{c}
     t \cdot (z_1,\dots,z_N)
     =
     (t_1^{\alpha_1}z_1,\dots,t_N^{\alpha_N}z_N)
  \end{array}$$
 induces an action on the Grassmannian $G(k, N)$ of $k$-vector spaces in $\C^N$. If we describe the elements in the Grassmannian by matrices, the fixed locus of the latter action is given by $(k\times N)$-matrices in row echelon form satisfying two extra conditions: there are no zero rows and there is only one nonzero element per row (the pivots). 
 
We will denote each of these fixed points by a $k$-tuple of natural numbers $(a_1\cdots a_k)$, with $a_1<\dots<a_k$. Note that the elements $e_{a_1}, e_{a_2},\dots, e_{a_k}$ span a $k$-plane which is a fixed point of the torus action on $G(k,N)$, so the number of fixed points in $G(k,N)$ is $N\choose k$.

$$
\begin{array}{l}\left(\begin{array}{cccccccccccc}
    0 & 0 &1 & 0                                 &\multicolumn{6}{c}  \cdotfill & 0  &0\\
    0 & \multicolumn{3}{c}  \cdotfill  & 1 & 0  &0& \multicolumn{4}{c}  \cdotfill   &0  \\
        & \vdots                                      & &  &  &  & \vdots &  &  &    &  & \vdots \\
     0 & \multicolumn{5}{c}  \cdotfill  & 0 & \multicolumn{3}{c}  \cdotfill  & 1 &  0 
\end{array}\right)
\\
\begin{picture}(0, 0)\thicklines
  \put(1, 0.1){\vector(-1,0){0.40}} 
   \put(0.5, 0){$|$}
   \put(1.05,.05){\tiny{$a_1$}}
   \put(1.35, 0.1){\vector(1,0){0.40}} 
   \put(1.75, 0){$|$} 
\end{picture}
\\
\begin{picture}(0, 0)\thicklines
  \put(1.5, 0.1){\vector(-1,0){0.90}} 
   \put(0.5, 0){$|$}
   \put(1.6,0){\tiny{$a_2$}}
   \put(1.95, 0.1){\vector(1,0){0.9}} 
   \put(2.85, 0){$|$} 
\end{picture}\\

\begin{picture}(0, 0)\thicklines
  \put(2.725, 0.1){\vector(-1,0){2.1}} 
   \put(0.5, 0){$|$}
   \put(3,0){\tiny{$a_k$}}
   \put(3.4, 0.1){\vector(1,0){2.1}} 
   \put(5.5, 0){$|$} 
\end{picture}\\

\\ 
\end{array} 
$$ 
\subsection{Additional weights}
For $1\leq m\leq n+1$, let us denote by $V^{(m)}$ a vector subspace of $\C^{n+1}$ spanned by $n$ of the coordinate vectors, $V^{(m)}=\mathrm{Span}(e_1, e_2,\dots,\widehat{e}_m,\dots,e_{n+1})$. Let us also define $V_{(m)}$ to be the quotient vector space given by $\C^{n+1}/\{e_m\}$, with projection $p_{m}: \C^{n+1} \rightarrow V_{(m)}$.

The inclusion $V^{(m)}\rightarrow\C^{n+1}$ induces the inclusion map
$$\iota_m:G(k,V^{(m)})\hookrightarrow G(k,\C^{n+1})=G(k,n+1)$$ defined by sending $\Lambda\subset V^{(m)}$ to $\Lambda\subset \C^{n+1}$, and the injective map
$$\kappa_m:G(k,V_{(m)})\rightarrow G(k+1,n+1)$$ obtained by mapping a $k$-space $\Lambda\subset V_{(m)}$ to $p_m^{-1}(\Lambda)\subset \C^{n+1}$.

\begin{pr}\label{kn+1}
 Let $C$ be a genus 0 curve with at worst nodes as singularities. Let $f: C \rightarrow G(k,V^{(m)})$ be a torus-invariant genus 0 stable
  map. Then the inclusion $\iota_m$ induces
  an injective map $$H^0(C,f^*T_{G(k,V^{(m)})}) \rightarrow
  H^0(C,f^*\iota_m^*T_{G(k, n+1)})$$ whose cokernel is given by
  the following list of weights
\begin{equation}\label{p21}[ \alpha_{m}-\alpha_{a_1},\dots, \alpha_{m}-\alpha_{a_{k-1}},
\{\alpha_{m}-(\frac{s}{d_i}\alpha_{b_u}+\frac{t}{d_i}\alpha_{b_v})\}_{0\leq
  s,t\leq d_i, s+t=d_i} ]\end{equation} for every irreducible component of $C$ mapping with degree $d_i$ to the line joining $(a_1\cdots a_{k-1}b_u)$ and $(a_1\cdots a_{k-1}b_v)$, with the following weights removed
\begin{align}\label{2p21}[\alpha_{m}-\alpha_{a_{1}},\dots,\alpha_{m}-\alpha_{a_{k}} ]\end{align}
for every node of $C$ mapping to the point $(a_1\cdots a_k)$.
\end{pr}
\begin{proof}
The exact sequence
$$ 0\longrightarrow T_{G(k,V^{(m)})} \longrightarrow \iota_m^*T_{G(k,n+1)}
\longrightarrow N \longrightarrow 0,$$
where $N=\iota_m^*T_{G(k,n+1)}/T_{G(k,V^{(m)})}$, gives rise to an exact sequence in cohomology. Using the vanishing $H^1(C, f^*T_{G(k,V^{(m)})})=0$,
\cite[Lemma 10]{FP97}, it reduces to the following short exact sequence
$$0\longrightarrow H^0(C, f^*T_{G(k,V^{(m)})})\longrightarrow H^0(C,
f^*\iota_m^*T_{G(k,n+1)})\longrightarrow H^0(C,f^*N)\longrightarrow 0.$$

The normal bundle $N$ equals $S^*\otimes \iota_m^*\widetilde{Q}/Q$, where
$S$ is the universal subbundle on $G(k,V^{(m)})$, and $Q$ and $\widetilde{Q}$ are the
quotient subbundles of $G(k,V^{(m)})$ and $G(k,n+1)$ respectively.

Let $\phi: \widehat{C}\rightarrow C$ be the normalization. So, the following sequence is exact
$$0\longrightarrow H^0(C,f^*N) \longrightarrow H^0(\widehat
{C},\phi^*f^*N)\longrightarrow  \oplus N_{p_i}\longrightarrow
H^1(C,f^*N)$$ where $p_i$ denote the images in $G(k,V^{(m)})$ of the nodes
of $C$.

Since $S^*$ is globally generated, we have $H^1(C,f^*S^*)=0$. The
vanishing of $H^1(C,f^*N)$ follows from that of $H^1(C,f^*S^*)$, since
$\iota_m^*\widetilde{Q}/Q$ is a trivial line bundle.

The list of weights (\ref{p21}) corresponds to $H^0(\widehat
C,\phi^*f^*N)$, whereas the term $\oplus N_{p_i}$ gives the list (\ref{2p21}) of
weights to be removed.
\end{proof}

\begin{cor}\label{ckn+1}
  The map $$T_{\smap_{0,0}(G(k,V^{(m)}),d),(C,f)}\rightarrow T_{\smap_{0,0}(G(k,n+1),d),(C,\iota_m\circ f)},$$ induced by the inclusion $\iota_m$ above, is injective and the cokernel is given by exactly  the weights listed above at the point $(C,f)$ in $\smap_{0,0}(G(k,V^{(m)}),d)$.
\end{cor}

If we now consider the inclusion map $\kappa_m:G(k,V_{(m)})\rightarrow G(k+1,n+1)$ defined above, we obtain the following

\begin{pr}\label{k+1n+1}
Let us consider
  $f: C\rightarrow G(k,V_{(m)})$ a torus-invariant genus 0 stable map. Then
  the inclusion $\kappa_m:G(k,V_{(m)})\rightarrow G(k+1,n+1)$ induces an injective
  map $$H^0(C,f^*T_{G(k,V_{(m)})}) \rightarrow H^0(C,f^*\kappa_m^*T_{G(k+1,n+1)})$$
  whose cokernel is given by the following list of weights
\begin{align}\label{p23}[ \alpha_{a'_1}-\alpha_{m},\dots, \alpha_{a'_{n-k-1}}-\alpha_{m},
\{(\frac{s}{d_i}\alpha_{b_u}+\frac{t}{d_i}\alpha_{b_v})-\alpha_{m}\}_{0\leq
  s,t\leq d_i, s+t=d_i} ]\end{align} for every irreducible component of $C$
mapping with degree $d_i$ to the line joining
$(a_1\cdots a_{k-1}b_u)$ and $(a_1\cdots a_{k-1}b_v)$, with the
following weights removed
\begin{align}\label{2p23}[\alpha_{a'_1}-\alpha_{m},\dots,\alpha_{a'_{n-k}}-\alpha_{m}]\end{align}
for every node of $C$ mapping to the point $(a_1\cdots a_k)$, and where 
$(a'_1\cdots a'_{n-k})$ denotes the complement of $(a_1\cdots a_k)$.

\end{pr}

\begin{proof}  Similarly to the proof of the previous proposition, an exact
  sequence in cohomology arises from the exact sequence
$$ 0\longrightarrow T_{G(k,V_{(m)})} \longrightarrow \kappa_m^*T_{G(k+1,n+1)}
\longrightarrow N \longrightarrow 0,$$
where $N=\kappa_m^*T_{G(k+1,n+1)}/T_{G(k,V_{(m)})}$.  The vanishing of $H^1(C,
f^*T_{G(k,V_{(m)})})$, as before, reduces it to the following short exact
sequence
$$0\longrightarrow H^0(C, f^*T_{G(k,V_{(m)})})\longrightarrow H^0(C,
f^*\kappa_m^*T_{G(k+1, n+1)})\longrightarrow H^0(C,f^*N)\longrightarrow 0$$

The normal bundle $N$ equals $(\ker(\widetilde{S}\rightarrow
S))^*\otimes Q$, where $Q$ is the quotient subbundle, and $S$ and
$\widetilde{S}$ are the universal subbundles of $G(k,V_{(m)})$ and
$G(k+1, n+1)$ respectively.

Let $\phi: \widehat{C}\rightarrow C$ be the normalization. Then, $$0\longrightarrow H^0(C,f^*N) \longrightarrow H^0(\widehat
C,\phi^*f^*N)\longrightarrow  \oplus N_{p_i}\longrightarrow
H^1(C,f^*N)$$ is an exact sequence, where $p_i$ denote the images in $G(k,V_{(m)})$ of the nodes
of $C$.
Since $Q$ is globally generated, we have $H^1(C,f^*Q)=0$. The
vanishing of $H^1(C,f^*N)$ follows from that of $H^1(C,f^*Q)$, since
$\ker(\widetilde{S}\rightarrow S)$ is a trivial line bundle.

The list of weights (\ref{p23}) corresponds to $H^0(\widehat
C,\phi^*f^*N)$, whereas the term $\oplus N_{p_i}$ gives the list (\ref{2p23}) of
weights to be removed.
\end{proof}

\begin{cor}\label{ck+1n+1}
  The map $$T_{\smap_{0,0}(G(k,V_{(m)}),d),(C,f)}\rightarrow
  T_{\smap_{0,0}(G(k+1,n+1),d),(C,\kappa_m\circ f)},$$
  induced by the inclusion $\kappa_m$, is injective and the
  cokernel is given by exactly the weights listed above, at the point
  $(C,f)$ in $\smap_{0,0}(G(k,V_{(m)}),d)$.
\end{cor}
\section{Poincar\'e polynomials of small degree stable maps}\label{poincare}

As a consequence of the work by Oprea in \cite{Op06}, the Betti numbers of the moduli stack of zero pointed genus zero degree $d$ stable maps to the Grassmannian can be computed from the fixed loci in the manner of the usual Bia{\l}ynicki-Birula decomposition \cite{BB73,BB76}. To describe the fixed loci of the action on $\smap\stog$ induced by that on the Grassmannian, we use trees of irreducible
rational curves whose labels on the vertices are tuples of natural numbers (as in Section \ref{loci}), and whose labels on the irreducible components indicate the degree of the map into the target variety. After this, our problem amounts to organizing the combinatorial data involved.

\begin{remark} In \cite[\S 3.2]{Kon95}, Kontsevich gives a description of  the fixed locus of the moduli space of degree $d$ stable maps to the projective space $\p^n$ under the action of $(\C^*)^{n+1}$. To keep
track of the combinatorial data arising from this description, the
fixed points are written as labelled connected graphs. He also gives complete formulas for the
weights of the fixed loci, which are displayed in terms of combinatorial data.
\end{remark}
\subsection{The degree 2 stable maps case}\label{deg2}
In this subsection we give the expression for the Poincar\'e polynomial of the moduli space $\smap_{0,0}(G(k, n),2)$.

Let us first recall the generating function for the Betti numbers of the classical Grassmannian $G(k, n)$:
$$
P_{G(k,n)}(q)=P_G(q)=\frac{\prod_{i=1}^{n}(1-q^i)}{\prod_{i=1}^{n-k}(1-q^i)\prod^{k}_{i=1}(1-q^i)}.
$$
\setlength{\unitlength}{0.4cm}
\begin{table}
 \caption{Fixed loci for degree 2 stable maps}\label{table1}
\begin{tabular}{|c|c|c|c|}

\hline

\small Configuration &\small  Labelling & \small Strata in & \small Number\\ \hline

\begin{picture}(3,1.5)
\put(0,0.6){\line(1,0){2.7}}
\end{picture} &
\begin{picture}(3,1.5)

\put(0.2,0.6){\line(1,0){2.7}}
\put(-.2,1){\makebox(0,0){\tiny$(j)$}}
\put(1.55,1){\makebox(0,0){\tiny$2$}}
\put(3.2,1){\makebox(0,0){\tiny$(i)$}}
\end{picture} &  \raisebox{1ex}{\small$G(1, 2)$}& \raisebox{1ex}{\small$\frac{1}{2}\binom{n}{k} k(n-k)$}\\\hline

\multirow{4}{*}{
\raisebox{-17ex}{
\begin{picture}(3,5)
\put(0,4.5){\line(1,0){2.7}}
\put(0.2,4.7){\line(0,-1){2.7}}
\end{picture}}} & 
\begin{picture}(3,3.7)
\put(0.2,2.8){\line(1,0){2.7}}
\put(0.4,3){\line(0,-1){2.7}}
\put(0.1,1.7){\makebox(0,0){\tiny$1$}}
\put(0,3.2){\makebox(0,0){\tiny$(j)$}}
\put(1.55,3.2){\makebox(0,0){\tiny$1$}}
\put(3.1,3.2){\makebox(0,0){\tiny$(i)$}}
\put(0,0.1){\makebox(0,0){\tiny$(i)$}} 
\end{picture} & \raisebox{3.5ex}{\small $G(1, 2)$} &  \multirow{4}{*}{\raisebox{-10ex}{\small $\binom{n}{k}\binom{k(n-k)+1}{2}$}} \\\cline{2-3}
&
\begin{picture}(3,3.7)
\put(0.2,2.8){\line(1,0){2.7}}
\put(0.4,3){\line(0,-1){2.7}}
\put(0.1,1.7){\makebox(0,0){\tiny$1$}}
\put(0,3.2){\makebox(0,0){\tiny$(i)$}}
\put(1.55,3.2){\makebox(0,0){\tiny$1$}}
\put(3.1,3.2){\makebox(0,0){\tiny$(j)$}}
\put(0,0.1){\makebox(0,0){\tiny$(j')$}} 
\end{picture}
&\raisebox{3.5ex}{\small $G(1, 3)$}&\\\cline{2-3}
&
\begin{picture}(3,3.7)
\put(0.2,2.8){\line(1,0){2.7}}
\put(0.4,3){\line(0,-1){2.7}}
\put(0.1,1.7){\makebox(0,0){\tiny$1$}}
\put(0,3.2){\makebox(0,0){\tiny$(ij)$}}
\put(1.55,3.2){\makebox(0,0){\tiny$1$}}
\put(3,3.2){\makebox(0,0){\tiny$(jj')$}}
\put(0,0.1){\makebox(0,0){\tiny$(ij')$}} 
\end{picture}&\raisebox{3.5ex}{\small $G(2, 3)$}&\\\cline{2-3}
&
\begin{picture}(3,3.7)
\put(0.2,2.8){\line(1,0){2.7}}
\put(0.4,3){\line(0,-1){2.7}}
\put(0.1,1.7){\makebox(0,0){\tiny$1$}}
\put(0,3.2){\makebox(0,0){\tiny$(ij)$}}
\put(1.55,3.2){\makebox(0,0){\tiny$1$}}
\put(3,3.2){\makebox(0,0){\tiny$(ij')$}}
\put(0,0.1){\makebox(0,0){\tiny$(i'j)$}} 
\end{picture}
&\raisebox{3.5ex}{\small $G(2, 4)$}&\\
\hline
\end{tabular}
\end{table}

\begin{tm}
\label{gkn2}
Let X denote the moduli stack of stable maps
  $\smap\stogii$. Then its Poincar\'e polynomial is given by
\begin{eqnarray*}
 P_{X}(q)= \frac{(1-q^k)(1-q^{n-k})\left((1+q^n)(1+q^3)-q(1+q)(q^k+q^{n-k})\right)}{(1-q)^2(1-q^2)^2}P_G(q).
\end{eqnarray*}

\end{tm}

\begin{proof}
For degree 2 stable maps, all strata in the fixed point decomposition can be obtained from strata in copies of $G(1,2)$, $G(1,3)$, $G(2,3)$, and $G(2,4)$ and their embeddings into bigger Grassmannians. In Table \ref{table1} above, we list all possible configurations for the fixed loci of the action. (For simplicity, we omit the list of all possible inequality conditions on the indices $i,j,j'$, and $i'$. Each of the cases in the table splits up into different subcases.) 

The contributions of the different fixed points of the $\C^*$-action
on $\smap\stogii$ to the Poincar\'e polynomial are all computed in a
similar way. As an illustration of the method, we will compute the contribution of two of the (many) possible tree configurations. 

Firstly, let us consider the contribution to the fixed points of $\smap\stogii$
given by
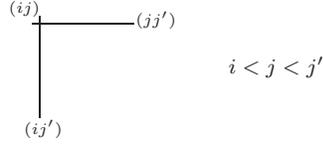
\setlength{\unitlength}{0.5cm}
\begin{center}
\begin{picture}(10,3.7)
\put(1,2.8){\line(1,0){2.7}}
\put(1.2,3){\line(0,-1){2.7}}
\put(0.8,3.2){\makebox(0,0){\tiny$(ij)$}}
\put(1.3,0){\makebox(0,0){\tiny$(ij')$}}
\put(4.3,2.9){\makebox(0,0){\tiny$(jj')$}}
\put(7.5,1.65){\makebox(0,0){\Small$i<j<j'$}}
\end{picture}
\end{center}

Taking $i=1$, $j=2$, and $j'=3$ on $G(2, 3)$ we get a list of 5
weights. Of these, 4 are positive, under the convention on weights
$\alpha_1\ll\alpha_2\ll\alpha_3$.

We consider now the process of extending to $G(k,n)$: let
\begin{align*}
S_0\cup T_0=\{1,\ldots,i-1\},\qquad & S_1\cup T_1=\{i+1,\ldots,j-1\}\\
S_2\cup T_2=\{j+1,\ldots,j'-1\},\qquad & S_3\cup T_3=\{j'+1,\ldots,n\}
\end{align*}
with $i_\nu=\#S_\nu$ and $j_\nu=\#T_\nu$ ($0\le\nu\le 3$) related by
\begin{align*}
&i_0+i_1+i_2+i_3=k-2,\quad &&i_0+j_0=i-1, \quad &&i_2+j_2=j'-j-1 \\
&j_0+j_1+j_2+j_3=n-k-1,\quad&&i_1+j_1=j-i-1, \quad &&i_3+j_3=n-j'.
\end{align*}

If we identify
\begin{align*}
G(2,3)&=G(2,\C\langle e_i,e_j,e_{j'}\rangle),\\
G(2,n-k+2)&=G(2,\C\langle e_\tau,\,\tau\in \{i,j,j'\}\cup T_0\cup T_1
\cup T_2\cup T_3\rangle),
\end{align*}
the embedding $G(2,3)\hookrightarrow G(2,n-k+2)$ gives rise, by Corollary \ref{ckn+1},  to the following added weights: 
$$[\alpha_t-\alpha_i, \alpha_t-\alpha_{j}, \alpha_t-\alpha_{j'}, \alpha_t-\alpha_{j'} : t\in T_0\cup T_1\cup T_2\cup T_3].$$

By Corollary \ref{ck+1n+1}, the further embedding in $G(k,n)$ gives rise to the additional weights: $$[\alpha_i-\alpha_s, \alpha_{j'}-\alpha_s, \alpha_{j}-\alpha_s, \alpha_t-\alpha_s :  s\in S_0\cup S_1\cup S_2\cup S_3, t\in T_0\cup T_1\cup T_2\cup T_3].$$

If we let $n_\nu=\#\{(s,t)\in S_\nu\times T_\nu\,|\,t>s\}$ then the
number of added weights under $G(2,3)\hookrightarrow G(k,n)$ that are
positive is
$$
  n_0+n_1+n_2+n_3+3i_0+2i_1+i_2+j_1+2j_2+4j_3+i_0j_1+i_0j_2+i_0j_3+i_1j_2+i_1j_3+i_2j_3.
$$
\begin{remark}\label{qbinom}
Sums over partitions of a finite set into the disjoint union of two
sets of given cardinality give rise to \emph{q}-binomial coefficients:
$$
\sum_{\stackrel{S\cup T={\{1,\dots,n\}}}{ \#S=k,\#T=n-k}}q^{\#\{(s,t)\in
  S\times T | t>s\}}={n\choose k}_q.$$
\end{remark}

Therefore the contribution to the Poincar\'e polynomial of
$\overline{{M}}_{0,0}(G(k,n),2)$ is
\begin{align}\label{cont23}\sum q^{4+3i_0+2i_1+i_2+j_1+2j_2+4j_3+i_0j_1+i_0j_2+i_0j_3+i_1j_2+i_1j_3+i_2j_3}
  \tbinom{i_0+j_0}{i_0}_q
  \tbinom{i_1+j_1}{i_1}_q \tbinom{i_2+j_2}{i_2}_q
  \tbinom{i_3+j_3}{i_3}_q,
\end{align}
where the sum is over $i_0+i_1+i_2+i_3=k-2$, $i_0+j_0=i-1$,
$i_1+j_1=j-i-1$, $i_2+j_2=j'-j-1$, $i_3+j_3=n-j'$.

To compute sums of this kind we make use of the identities on
$q$-binomial coefficients given in the Appendix. For example, applying the corresponding
identity to $\sum q^{(i_2+2)j_3}{i_2+j_2 \choose i_2}_q{i_3+j_3
  \choose i_3}_q$, the sum (\ref{cont23}) equals
\begin{align*}
  \sum&q^{4+3i_0+2i_1+i_2+j_1+2j'_2+i_0j_1+i_0j'_2+i_1j'_2}
  \cdot\tbinom{n-j+1}{i_2+i_3+2}_q \tbinom{i_0+j_0}{i_0}_q
  \tbinom{i_1+j_1}{i_1}_q\\
  &-\sum
  q^{7+4i_0+3i_1+2i_2+j_1+2j'_2+i_0j_1+i_0j'_2+i_1j'_2}\cdot\tbinom{n-j}{i_2+i_3+2}_q
  \tbinom{i_0+j_0}{i_0}_q \tbinom{i_1+j_1}{i_1}_q\ ,
\end{align*}
where $j'_2=j_2+j_3$, and the sums are over $i_0+i_1+i_2+i_3=k-2$,
$i_0+j_0=i-1$, $i_1+j_1=j-i-1$, and $i_2+i_3+j'_2=n-j-1$ and
$i_2+i_3+j'_2=n-j-2$ respectively.

We recognize that the sums above are equal to
$$q^4\sum_{i_0+i_1+i_2+i_3=k-2}q^{3i_0+2i_1+i_2}{\tbinom{n+1}{k+2}_q}-q^7\sum_{i_0+i_1+i_2+i_3=k-2}q^{4i_0+3i_1+2i_2}{\tbinom{n}{k+2}_q}.$$

Continuing applying identities on $q$-binomial coefficients leads to
the expression
\begin{align}
 \label{2cont23}q^4{\tbinom{k+1}{3}_q}{\tbinom{n+1}{k+2}_q}-q^7{\tbinom{k+2}{4}_q}{\tbinom{n}{k+2}_q}+q^8{\tbinom{k+1}{4}_q}{\tbinom{n}{k+2}_q}.\end{align}

The same process can be applied to any of the configurations on Table \ref{table1}. For example, another degree two map, fixed by the torus action, is that given by
the following tree of $\p^1$'s:
\setlength{\unitlength}{0.5cm}
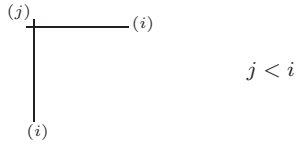
\begin{center}
\begin{picture}(10,3.7)
\put(1,2.8){\line(1,0){2.7}}
\put(1.2,3){\line(0,-1){2.7}}
\put(0.8,3.2){\makebox(0,0){\tiny$(j)$}}
\put(1.3,0){\makebox(0,0){\tiny$(i)$}}
\put(4.1,2.9){\makebox(0,0){\tiny$(i)$}}
\put(7.5,1.65){\makebox(0,0){\Small$j<i$}}
\end{picture}
\end{center}

Consider again the process of extending to $G(k, n)$: Corollaries \ref{ckn+1} and \ref{ck+1n+1} give rise to lists of additional weights. If we use again Remark \ref{qbinom} and the identities in the Appendix, we obtain that the contribution of this configuration to the Poincar\'e polynomial of $\overline{M}_{0,0}(G(k, n),2)$ is
\begin{align}
  \label{2cont12} 
  q^2
  \left(
    \tbinom{k+2}{3}_q
    \tbinom{n+1}{k+2}_q
    -
    q^2
    \tbinom{k+1}{3}_q
    \tbinom{n+1}{k+2}_q
    -
    q^2
    \tbinom{k+3}{4}_q
    \tbinom{n}{k+2}_q
    +
    q^3
    \tbinom{k+2}{4}_q
    \tbinom{n}{k+2}_q
  \right.  \\
  \left.
    +
    q^5
    \tbinom{k+2}{4}_q
    \tbinom{n}{k+2}_q
    -
    q^6
    \tbinom{k+1}{4}_q
    \tbinom{n}{k+2}_q
  \right). \nonumber
\end{align}

For each of the possible fixed point configurations described in Table \ref{table1}, we get different sums involving $q$-binomial coefficients, e.g. expressions like (\ref{2cont23}) and (\ref{2cont12}). Adding them up, we obtain the rational function in the statement of the Theorem.
\end{proof}

\subsection{Poincar\'e polynomials of degree 3 stable maps} \label{deg3}
In this subsection we present the generating function for
the Poincar\'e polynomial in the case of no pointed genus zero
degree three stable maps to the Grassmannian. The method used to organize the combinatorial data involved and prove Theorem \ref{gkn3} is the same as the one described for the case of degree 2 stable maps. We will soon see that the number of different configurations for the labelled connected graphs representing the fixed locus increases rapidly.

\begin{table}
\caption{Fixed loci for degree 3 stable maps I}\label{table2}

\begin{tabular}{|c|c|c|c|}

\hline

\small Configuration &\small  Labelling & \small Strata in & \small Number\\ \hline

\begin{picture}(3,1.5)
\put(0,0.6){\line(1,0){2.7}}
\end{picture} &
\begin{picture}(3,1.5)

\put(0.2,0.6){\line(1,0){2.7}}
\put(-.2,1){\makebox(0,0){\tiny$(j)$}}
\put(1.55,1){\makebox(0,0){\tiny$3$}}
\put(3.2,1){\makebox(0,0){\tiny$(i)$}}
\end{picture} &  \raisebox{1ex}{\small$G(1,2)$}& \raisebox{1ex}{\small$\frac{1}{2}\binom{n}{k} k(n-k)$}\\\hline

\multirow{4}{*}{
\raisebox{-17ex}{
\begin{picture}(3,5)
\put(0,4.5){\line(1,0){2.7}}
\put(0.2,4.7){\line(0,-1){2.7}}
\end{picture}}} & 
\begin{picture}(3,3.7)
\put(0.2,2.8){\line(1,0){2.7}}
\put(0.4,3){\line(0,-1){2.7}}
\put(0.1,1.7){\makebox(0,0){\tiny$1$}}
\put(0,3.2){\makebox(0,0){\tiny$(j)$}}
\put(1.55,3.2){\makebox(0,0){\tiny$2$}}
\put(3.1,3.2){\makebox(0,0){\tiny$(i)$}}
\put(0,0.1){\makebox(0,0){\tiny$(i)$}} 
\end{picture} & \raisebox{3.5ex}{\small $G(1,2)$} &  \multirow{4}{*}{\raisebox{-10ex}{\small$\binom{n}{k} k^2(n-k)^2$}} \\\cline{2-3}
&
\begin{picture}(3,3.7)
\put(0.2,2.8){\line(1,0){2.7}}
\put(0.4,3){\line(0,-1){2.7}}
\put(0.1,1.7){\makebox(0,0){\tiny$1$}}
\put(0,3.2){\makebox(0,0){\tiny$(i)$}}
\put(1.55,3.2){\makebox(0,0){\tiny$2$}}
\put(3.1,3.2){\makebox(0,0){\tiny$(j)$}}
\put(0,0.1){\makebox(0,0){\tiny$(j')$}} 
\end{picture}
&\raisebox{3.5ex}{\small $G(1,3)$}&\\\cline{2-3}
&
\begin{picture}(3,3.7)
\put(0.2,2.8){\line(1,0){2.7}}
\put(0.4,3){\line(0,-1){2.7}}
\put(0.1,1.7){\makebox(0,0){\tiny$1$}}
\put(0,3.2){\makebox(0,0){\tiny$(ij)$}}
\put(1.55,3.2){\makebox(0,0){\tiny$2$}}
\put(3,3.2){\makebox(0,0){\tiny$(jj')$}}
\put(0,0.1){\makebox(0,0){\tiny$(ij')$}} 
\end{picture}&\raisebox{3.5ex}{\small $G(2,3)$}&\\\cline{2-3}
&
\begin{picture}(3,3.7)
\put(0.2,2.8){\line(1,0){2.7}}
\put(0.4,3){\line(0,-1){2.7}}
\put(0.1,1.7){\makebox(0,0){\tiny$1$}}
\put(0,3.2){\makebox(0,0){\tiny$(ij)$}}
\put(1.55,3.2){\makebox(0,0){\tiny$2$}}
\put(3,3.2){\makebox(0,0){\tiny$(ij')$}}
\put(0,0.1){\makebox(0,0){\tiny$(i'j)$}} 
\end{picture}
&\raisebox{3.5ex}{\small $G(2,4)$}&\\

\hline

\end{tabular}
\end{table}

\begin{tm}\label{gkn3}
Let X be the moduli stack of stable maps
  $\smap\stogiii$, then its Poincar\'e polynomial $P_X(q)$ is given by the following expression:

\begin{align*}
 (1-q^k)(1-q^{n-k})&  	\left(
                          			F_1(q)(1+q^{2n})
                          			+
	                          		F_2(q)q^2(q^{2k}+q^{2n-2k})
                         		\right. \\    
                        	     &\left.
         		                  		+
                   	       		(1+q)^2
                          				\left(
                            				F_3(q)q^n(1+q^2) 
                            				-
                            				F_4(q)q(1+q^n)(q^{k}+q^{n-k})
		                          		\right)  
                         		\right)\\
		                 &\phantom{,}\cdot \frac{P_{G}(q)}{(1-q)^{2}(1-q^2)^{3}(1-q^3)^{2}},
\end{align*}
where
\begin{eqnarray*}
&&F_1(q)=1+2q^2+3q^3+3q^4-q^5+q^6-3q^7-3q^8-2q^9-q^{11},\\ 
&& F_2(q)=1+6q+3q^2+2q^3-2q^4-3q^5-6q^6-q^7,\\
&&F_3(q)=1+5q^2+2q^3-2q^4-5q^5-q^7,\\
&&F_4(q)=2+3q^2+q^3-q^4-3q^5-2q^7.
\end{eqnarray*}
\end{tm}
\begin{proof}
We observe that for degree 3 stable maps all strata in the fixed point decomposition can be obtained from strata in copies of $G(1, 2)$, $G(1,3)$, $G(1,4)$, $G(2,3)$, $G(2,4)$, $G(2,5)$, $G(3,4)$, $G(3,5)$, and $G(3,6)$, and their embeddings into bigger Grassmannians. In Tables \ref{table2} and \ref{table3}, we list all possible configurations and labellings for fixed points of the torus action. To obtain the expression in the statement of the Theorem, we use the method in the proof of Theorem \ref{gkn2}, which relies on the results in Section \ref{torus}.

\begin{table}
\caption{Fixed loci for degree 3 stable maps II}\label{table3}
\begin{tabular}{|c|c|c||c|c|c|}
\hline
\small Configuration &\small  Labelling & \small Strata in &\small  Labelling & \small Strata in &\small Number \\ \hline
\multirow{4}{*}{
\raisebox{-22ex}{
\begin{picture}(3,5)
\put(0,4.5){\line(1,0){2.7}}
\put(0,4.7){\line(1,-1){2.5}}
\put(0.2,4.7){\line(0,-1){2.7}}
\end{picture}}} & 
\begin{picture}(3,3.7)
\put(0.2,2.8){\line(1,0){2.7}}
\put(0.2,3){\line(1,-1){2.5}}
\put(0.4,3){\line(0,-1){2.7}}
\put(0.1,1.7){\makebox(0,0){\tiny$1$}}
\put(0,3.2){\makebox(0,0){\tiny$(i)$}}
\put(1.55,3.2){\makebox(0,0){\tiny$1$}}
\put(3.1,3.2){\makebox(0,0){\tiny$(j)$}}
\put(3,0.2){\makebox(0,0){\tiny$(j)$}}
\put(1.9,1.8){\makebox(0,0){\tiny$1$}}
\put(0,0.1){\makebox(0,0){\tiny$(j)$}} 
\end{picture} & \raisebox{3.5ex}{\small $G(1,2)$} & 
\begin{picture}(3,3.7)
\put(0.2,2.8){\line(1,0){2.7}}
\put(0.2,3){\line(1,-1){2.5}}
\put(0.4,3){\line(0,-1){2.7}}
\put(0.1,1.7){\makebox(0,0){\tiny$1$}}
\put(0,3.2){\makebox(0,0){\tiny$(i)$}}
\put(1.55,3.2){\makebox(0,0){\tiny$1$}}
\put(3.1,3.2){\makebox(0,0){\tiny$(j)$}}
\put(3,0.2){\makebox(0,0){\tiny$(j)$}}
\put(1.9,1.8){\makebox(0,0){\tiny$1$}}
\put(0,0.1){\makebox(0,0){\tiny$(j')$}} 
\end{picture} & \raisebox{3.5ex}{\small $G(1,3)$}&  \multirow{4}{*}{\raisebox{-15ex}{\small$\binom{n}{k}\binom{k(n-k)+2}{3}$}}\\\cline{2-5}
&
\begin{picture}(3,3.7)
\put(0.2,2.8){\line(1,0){2.7}}
\put(0.2,3){\line(1,-1){2.5}}
\put(0.4,3){\line(0,-1){2.7}}
\put(0.1,1.7){\makebox(0,0){\tiny$1$}}
\put(0,3.2){\makebox(0,0){\tiny$(i)$}}
\put(1.55,3.2){\makebox(0,0){\tiny$1$}}
\put(3.1,3.2){\makebox(0,0){\tiny$(j)$}}
\put(3,0.2){\makebox(0,0){\tiny$(i')$}}
\put(1.9,1.8){\makebox(0,0){\tiny$1$}}
\put(0,0.1){\makebox(0,0){\tiny$(j')$}} 
\end{picture}
&\raisebox{3.5ex}{\small $G(1,4)$}& \begin{picture}(3,3.7)
\put(0.2,2.8){\line(1,0){2.7}}
\put(0.2,3){\line(1,-1){2.5}}
\put(0.4,3){\line(0,-1){2.7}}
\put(0.1,1.7){\makebox(0,0){\tiny$1$}}
\put(0,3.2){\makebox(0,0){\tiny$(ij)$}}
\put(1.55,3.2){\makebox(0,0){\tiny$1$}}
\put(3.1,3.2){\makebox(0,0){\tiny$(ij')$}}
\put(3,0.2){\makebox(0,0){\tiny$(ij')$}}
\put(1.9,1.8){\makebox(0,0){\tiny$1$}}
\put(0,0.1){\makebox(0,0){\tiny$(jj')$}} 
\end{picture} & \raisebox{3.5ex}{\small $G(2,3)$} &\\\cline{2-5}
&
\begin{picture}(3,3.7)
\put(0.2,2.8){\line(1,0){2.7}}
\put(0.2,3){\line(1,-1){2.5}}
\put(0.4,3){\line(0,-1){2.7}}
\put(0.1,1.7){\makebox(0,0){\tiny$1$}}
\put(0,3.2){\makebox(0,0){\tiny$(ij)$}}
\put(1.55,3.2){\makebox(0,0){\tiny$1$}}
\put(3.1,3.2){\makebox(0,0){\tiny$(ij')$}}
\put(3,0.2){\makebox(0,0){\tiny$(ij')$}}
\put(1.9,1.8){\makebox(0,0){\tiny$1$}}
\put(0,0.1){\makebox(0,0){\tiny$(ji')$}} 
\end{picture} & \raisebox{3.5ex}{\small $G(2,4)$}&\begin{picture}(3,3.7)
\put(0.2,2.8){\line(1,0){2.7}}
\put(0.2,3){\line(1,-1){2.5}}
\put(0.4,3){\line(0,-1){2.7}}
\put(0.1,1.7){\makebox(0,0){\tiny$1$}}
\put(0,3.2){\makebox(0,0){\tiny$(ij)$}}
\put(1.55,3.2){\makebox(0,0){\tiny$1$}}
\put(3,3.2){\makebox(0,0){\tiny$(ij')$}}
\put(3,0.2){\makebox(0,0){\tiny$(ii')$}}
\put(1.9,1.8){\makebox(0,0){\tiny$1$}}
\put(0,0.1){\makebox(0,0){\tiny$(jj')$}} 
\end{picture}&\raisebox{3.5ex}{\small $G(2,4)$}&\\\cline{2-5}
&
\begin{picture}(3,3.7)
\put(0.2,2.8){\line(1,0){2.7}}
\put(0.2,3){\line(1,-1){2.5}}
\put(0.4,3){\line(0,-1){2.7}}
\put(0.1,1.7){\makebox(0,0){\tiny$1$}}
\put(0,3.2){\makebox(0,0){\tiny$(ij)$}}
\put(1.55,3.2){\makebox(0,0){\tiny$1$}}
\put(3,3.2){\makebox(0,0){\tiny$(ij')$}}
\put(3,0.2){\makebox(0,0){\tiny$(ii')$}}
\put(1.9,1.8){\makebox(0,0){\tiny$1$}}
\put(0,0.1){\makebox(0,0){\tiny$(jl)$}} 
\end{picture}
&\raisebox{3.5ex}{\small $G(2,5)$}&\begin{picture}(3,3.7)
\put(0.2,2.8){\line(1,0){2.7}}
\put(0.2,3){\line(1,-1){2.5}}
\put(0.4,3){\line(0,-1){2.7}}
\put(0.1,1.7){\makebox(0,0){\tiny$1$}}
\put(0,3.2){\makebox(0,0){\tiny$(ijl)$}}
\put(1.55,3.2){\makebox(0,0){\tiny$1$}}
\put(3,3.2){\makebox(0,0){\tiny$(ijj')$}}
\put(3,0.2){\makebox(0,0){\tiny$(ilj')$}}
\put(1.9,1.8){\makebox(0,0){\tiny$1$}}
\put(0,0.1){\makebox(0,0){\tiny$(jlj')$}} 
\end{picture}
&\raisebox{3.5ex}{\small $G(3,4)$}&\\

\cline{2-5}
&
\begin{picture}(3,3.7)
\put(0.2,2.8){\line(1,0){2.7}}
\put(0.2,3){\line(1,-1){2.5}}
\put(0.4,3){\line(0,-1){2.7}}
\put(0.1,1.7){\makebox(0,0){\tiny$1$}}
\put(0,3.2){\makebox(0,0){\tiny$(ijl)$}}
\put(1.55,3.2){\makebox(0,0){\tiny$1$}}
\put(3,3.2){\makebox(0,0){\tiny$(ijj')$}}
\put(3,0.2){\makebox(0,0){\tiny$(ilj')$}}
\put(1.9,1.8){\makebox(0,0){\tiny$1$}}
\put(0,0.1){\makebox(0,0){\tiny$(jli')$}} 
\end{picture}
&\raisebox{3.5ex}{\small $G(3,5)$}&\begin{picture}(3,3.7)
\put(0.2,2.8){\line(1,0){2.7}}
\put(0.2,3){\line(1,-1){2.5}}
\put(0.4,3){\line(0,-1){2.7}}
\put(0.1,1.7){\makebox(0,0){\tiny$1$}}
\put(0,3.2){\makebox(0,0){\tiny$(ijl)$}}
\put(1.55,3.2){\makebox(0,0){\tiny$1$}}
\put(3,3.2){\makebox(0,0){\tiny$(iji')$}}
\put(3,0.2){\makebox(0,0){\tiny$(ilj')$}}
\put(1.9,1.8){\makebox(0,0){\tiny$1$}}
\put(0,0.1){\makebox(0,0){\tiny$(jll')$}} 
\end{picture}
&\raisebox{3.5ex}{\small $G(3,6)$}&\\

\hline
\hline
\multirow{4}{*}{
\raisebox{-35ex}{
\begin{picture}(3,3.7)
\put(0,2.8){\line(1,0){2.7}}
\put(0.2,3){\line(0,-1){2.7}}
\put(0,0.5){\line(1,0){2.7}}
\end{picture}}} & 
\begin{picture}(3,3.7)
\put(0,2.8){\line(1,0){2.7}}
\put(0.2,3){\line(0,-1){2.7}}
\put(0,0.5){\line(1,0){2.7}}
\put(0,1.7){\makebox(0,0){\tiny$1$}}
\put(0,3.2){\makebox(0,0){\tiny$(i)$}}
\put(1.55,3.2){\makebox(0,0){\tiny$1$}}
\put(3.1,3.2){\makebox(0,0){\tiny$(j)$}}
\put(3,0.2){\makebox(0,0){\tiny$(i)$}}
\put(1.55,0.1){\makebox(0,0){\tiny$1$}}
\put(0,0.1){\makebox(0,0){\tiny$(j)$}} 
\end{picture} & \raisebox{3.5ex}{\small $G(1,2)$} & 
\begin{picture}(3,3.7)
\put(0,2.8){\line(1,0){2.7}}
\put(0.2,3){\line(0,-1){2.7}}
\put(0,0.5){\line(1,0){2.7}}
\put(0,1.7){\makebox(0,0){\tiny$1$}}
\put(0,3.2){\makebox(0,0){\tiny$(i)$}}
\put(1.55,3.2){\makebox(0,0){\tiny$1$}}
\put(3.1,3.2){\makebox(0,0){\tiny$(j)$}}
\put(3,0.2){\makebox(0,0){\tiny$(j)$}}
\put(1.55,0.1){\makebox(0,0){\tiny$1$}}
\put(0,0.1){\makebox(0,0){\tiny$(j')$}} 
\end{picture} & \raisebox{3.5ex}{\small $G(1,3)$}&  \multirow{7}{*}{\raisebox{-32ex}{\small$ \frac{1}{2}\binom{n}{k} k^3(n-k)^3$}}\\\cline{2-5}
&
\begin{picture}(3,3.7)
\put(0,2.8){\line(1,0){2.7}}
\put(0.2,3){\line(0,-1){2.7}}
\put(0,0.5){\line(1,0){2.7}}
\put(0,1.7){\makebox(0,0){\tiny$1$}}
\put(0,3.2){\makebox(0,0){\tiny$(i)$}}
\put(1.55,3.2){\makebox(0,0){\tiny$1$}}
\put(3.1,3.2){\makebox(0,0){\tiny$(j)$}}
\put(3,0.2){\makebox(0,0){\tiny$(i')$}}
\put(1.55,0.1){\makebox(0,0){\tiny$1$}}
\put(0,0.1){\makebox(0,0){\tiny$(j')$}} 
\end{picture}
&\raisebox{3.5ex}{\small $G(1,4)$}&\begin{picture}(3,3.7)
\put(0,2.8){\line(1,0){2.7}}
\put(0.2,3){\line(0,-1){2.7}}
\put(0,0.5){\line(1,0){2.7}}
\put(0,1.7){\makebox(0,0){\tiny$1$}}
\put(0,3.2){\makebox(0,0){\tiny$(i)$}}
\put(1.55,3.2){\makebox(0,0){\tiny$1$}}
\put(3.1,3.2){\makebox(0,0){\tiny$(j)$}}
\put(3,0.2){\makebox(0,0){\tiny$(j')$}}
\put(1.55,0.1){\makebox(0,0){\tiny$1$}}
\put(0,0.1){\makebox(0,0){\tiny$(j)$}} 
\end{picture} & \raisebox{3.5ex}{\small $G(1,3)$} &\\\cline{2-5}
&
\begin{picture}(3,3.7)
\put(0,2.8){\line(1,0){2.7}}
\put(0.2,3){\line(0,-1){2.7}}
\put(0,0.5){\line(1,0){2.7}}
\put(0,1.7){\makebox(0,0){\tiny$1$}}
\put(0,3.2){\makebox(0,0){\tiny$(ij)$}}
\put(1.55,3.2){\makebox(0,0){\tiny$1$}}
\put(3.1,3.2){\makebox(0,0){\tiny$(ij')$}}
\put(3,0.2){\makebox(0,0){\tiny$(jj')$}}
\put(1.55,0.1){\makebox(0,0){\tiny$1$}}
\put(0,0.1){\makebox(0,0){\tiny$(ij')$}} 
\end{picture} & \raisebox{3.5ex}{\small $G(2,3)$}&\begin{picture}(3,3.7)
\put(0,2.8){\line(1,0){2.7}}
\put(0.2,3){\line(0,-1){2.7}}
\put(0,0.5){\line(1,0){2.7}}\put(0,1.7){\makebox(0,0){\tiny$1$}}
\put(0,3.2){\makebox(0,0){\tiny$(ij)$}}
\put(1.55,3.2){\makebox(0,0){\tiny$1$}}
\put(3.1,3.2){\makebox(0,0){\tiny$(ij')$}}
\put(3,0.2){\makebox(0,0){\tiny$(j'i')$}}
\put(1.55,0.1){\makebox(0,0){\tiny$1$}}
\put(0,0.1){\makebox(0,0){\tiny$(ij')$}} 
\end{picture} & \raisebox{3.5ex}{\small $G(2,4)$}&\\\cline{2-5}
&
\begin{picture}(3,3.7)
\put(0,2.8){\line(1,0){2.7}}
\put(0.2,3){\line(0,-1){2.7}}
\put(0,0.5){\line(1,0){2.7}}
\put(0,1.7){\makebox(0,0){\tiny$1$}}
\put(0,3.2){\makebox(0,0){\tiny$(ij)$}}
\put(1.55,3.2){\makebox(0,0){\tiny$1$}}
\put(3,3.2){\makebox(0,0){\tiny$(ij')$}}
\put(3,0.2){\makebox(0,0){\tiny$(ij')$}}
\put(1.55,0.1){\makebox(0,0){\tiny$1$}}
\put(0,0.1){\makebox(0,0){\tiny$(jj')$}} 
\end{picture}&\raisebox{3.5ex}{\small $G(2,3)$}&\begin{picture}(3,3.7)
\put(0,2.8){\line(1,0){2.7}}
\put(0.2,3){\line(0,-1){2.7}}
\put(0,0.5){\line(1,0){2.7}}\put(0,1.7){\makebox(0,0){\tiny$1$}}
\put(0,3.2){\makebox(0,0){\tiny$(ij)$}}
\put(1.55,3.2){\makebox(0,0){\tiny$1$}}
\put(3,3.2){\makebox(0,0){\tiny$(ij')$}}
\put(3,0.2){\makebox(0,0){\tiny$(ji')$}}
\put(1.55,0.1){\makebox(0,0){\tiny$1$}}
\put(0,0.1){\makebox(0,0){\tiny$(ii')$}} 
\end{picture}
&\raisebox{3.5ex}{\small $G(2,4)$}&\\\cline{2-5}
&
\begin{picture}(3,3.7)
\put(0,2.8){\line(1,0){2.7}}
\put(0.2,3){\line(0,-1){2.7}}
\put(0,0.5){\line(1,0){2.7}}
\put(0,1.7){\makebox(0,0){\tiny$1$}}
\put(0,3.2){\makebox(0,0){\tiny$(ij)$}}
\put(1.55,3.2){\makebox(0,0){\tiny$1$}}
\put(3,3.2){\makebox(0,0){\tiny$(ij')$}}
\put(3,0.2){\makebox(0,0){\tiny$(j'i')$}}
\put(1.55,0.1){\makebox(0,0){\tiny$1$}}
\put(0,0.1){\makebox(0,0){\tiny$(ii')$}} 
\end{picture}&\raisebox{3.5ex}{\small $G(2,4)$}&\begin{picture}(3,3.7)
\put(0,2.8){\line(1,0){2.7}}
\put(0.2,3){\line(0,-1){2.7}}
\put(0,0.5){\line(1,0){2.7}}
\put(0,1.7){\makebox(0,0){\tiny$1$}}
\put(0,3.2){\makebox(0,0){\tiny$(ij)$}}
\put(1.55,3.2){\makebox(0,0){\tiny$1$}}
\put(3,3.2){\makebox(0,0){\tiny$(ij')$}}
\put(3,0.2){\makebox(0,0){\tiny$(i'l)$}}
\put(1.55,0.1){\makebox(0,0){\tiny$1$}}
\put(0,0.1){\makebox(0,0){\tiny$(ii')$}} 
\end{picture}
&\raisebox{3.5ex}{\small $G(2,5)$}&\\\cline{2-5}
&\begin{picture}(3,3.7)
\put(0,2.8){\line(1,0){2.7}}
\put(0.2,3){\line(0,-1){2.7}}
\put(0,0.5){\line(1,0){2.7}}
\put(0,1.7){\makebox(0,0){\tiny$1$}}
\put(0,3.2){\makebox(0,0){\tiny$(ij)$}}
\put(1.55,3.2){\makebox(0,0){\tiny$1$}}
\put(3,3.2){\makebox(0,0){\tiny$(ij')$}}
\put(3,0.2){\makebox(0,0){\tiny$(j'i')$}}
\put(1.55,0.1){\makebox(0,0){\tiny$1$}}
\put(0,0.1){\makebox(0,0){\tiny$(ji')$}} 
\end{picture}&\raisebox{3.5ex}{\small $G(2,4)$}
&\begin{picture}(3,3.7)
\put(0,2.8){\line(1,0){2.7}}
\put(0.2,3){\line(0,-1){2.7}}
\put(0,0.5){\line(1,0){2.7}}
\put(0,1.7){\makebox(0,0){\tiny$1$}}
\put(0,3.2){\makebox(0,0){\tiny$(ij)$}}
\put(1.55,3.2){\makebox(0,0){\tiny$1$}}
\put(3,3.2){\makebox(0,0){\tiny$(ij')$}}
\put(3,0.2){\makebox(0,0){\tiny$(ii')$}}
\put(1.55,0.1){\makebox(0,0){\tiny$1$}}
\put(0,0.1){\makebox(0,0){\tiny$(ji')$}} 
\end{picture}
&\raisebox{3.5ex}{\small $G(2,4)$}&\\\cline{2-5}
&
\begin{picture}(3,3.7)
\put(0,2.8){\line(1,0){2.7}}
\put(0.2,3){\line(0,-1){2.7}}
\put(0,0.5){\line(1,0){2.7}}
\put(0,1.7){\makebox(0,0){\tiny$1$}}
\put(0,3.2){\makebox(0,0){\tiny$(ij)$}}
\put(1.55,3.2){\makebox(0,0){\tiny$1$}}
\put(3,3.2){\makebox(0,0){\tiny$(ij')$}}
\put(3,0.2){\makebox(0,0){\tiny$(i'l)$}}
\put(1.55,0.1){\makebox(0,0){\tiny$1$}}
\put(0,0.1){\makebox(0,0){\tiny$(ji')$}} 
\end{picture}
&\raisebox{3.5ex}{\small $G(2,5)$}&&&\\

\cline{2-5}
&
\begin{picture}(3,3.7)
\put(0,2.8){\line(1,0){2.7}}
\put(0.2,3){\line(0,-1){2.7}}
\put(0,0.5){\line(1,0){2.7}}
\put(0,1.7){\makebox(0,0){\tiny$1$}}
\put(0,3.2){\makebox(0,0){\tiny$(ijl)$}}
\put(1.55,3.2){\makebox(0,0){\tiny$1$}}
\put(3,3.2){\makebox(0,0){\tiny$(ijl')$}}
\put(3,0.2){\makebox(0,0){\tiny$(jll')$}}
\put(1.55,0.1){\makebox(0,0){\tiny$1$}}
\put(0,0.1){\makebox(0,0){\tiny$(ill')$}} 
\end{picture}&\raisebox{3.5ex}{\small $G(3,4)$}&\begin{picture}(3,3.7)
\put(0,2.8){\line(1,0){2.7}}
\put(0.2,3){\line(0,-1){2.7}}
\put(0,0.5){\line(1,0){2.7}}
\put(0,1.7){\makebox(0,0){\tiny$1$}}
\put(0,3.2){\makebox(0,0){\tiny$(ijl)$}}
\put(1.55,3.2){\makebox(0,0){\tiny$1$}}
\put(3,3.2){\makebox(0,0){\tiny$(ijl')$}}
\put(3,0.2){\makebox(0,0){\tiny$(j'll')$}}
\put(1.55,0.1){\makebox(0,0){\tiny$1$}}
\put(0,0.1){\makebox(0,0){\tiny$(ill')$}} 
\end{picture}
&\raisebox{3.5ex}{\small $G(3,5)$}&\\

\cline{2-5}
&
\begin{picture}(3,3.7)
\put(0,2.8){\line(1,0){2.7}}
\put(0.2,3){\line(0,-1){2.7}}
\put(0,0.5){\line(1,0){2.7}}
\put(0,1.7){\makebox(0,0){\tiny$1$}}
\put(0,3.2){\makebox(0,0){\tiny$(ijl)$}}
\put(1.55,3.2){\makebox(0,0){\tiny$1$}}
\put(3,3.2){\makebox(0,0){\tiny$(ijl')$}}
\put(3,0.2){\makebox(0,0){\tiny$(j'll')$}}
\put(1.55,0.1){\makebox(0,0){\tiny$1$}}
\put(0,0.1){\makebox(0,0){\tiny$(ij'l)$}} 
\end{picture}&\raisebox{3.5ex}{\small $G(3,5)$}&\begin{picture}(3,3.7)
\put(0,2.8){\line(1,0){2.7}}
\put(0.2,3){\line(0,-1){2.7}}
\put(0,0.5){\line(1,0){2.7}}
\put(0,1.7){\makebox(0,0){\tiny$1$}}
\put(0,3.2){\makebox(0,0){\tiny$(ijl)$}}
\put(1.55,3.2){\makebox(0,0){\tiny$1$}}
\put(3,3.2){\makebox(0,0){\tiny$(ijl')$}}
\put(3,0.2){\makebox(0,0){\tiny$(i'j'l)$}}
\put(1.55,0.1){\makebox(0,0){\tiny$1$}}
\put(0,0.1){\makebox(0,0){\tiny$(ij'l)$}} 
\end{picture}
&\raisebox{3.5ex}{\small $G(3,6)$}&\\
\hline
\end{tabular}
\end{table}
\end{proof}

The following is a table of Betti numbers for some values of $k$ and $n$. The sequences of Betti numbers in the first two rows agree with the results in \cite{GP06}. Moreover, in degree 2, for $k=1$ and arbitrary $n$, one can check that our expression agrees with that due to Getzler and Pandharipande and presented explicitly in \cite[(39)]{BO03}.
\begin{equation*}
\begin{tabular}{|c|c|l|} \hline
  $k$ &$n$&Betti numbers of $\smap_{0,0}(G(k,n),3)$ \\
  \hline $1$ &$3$&$(1,2,5,7,9,7,5,2,1)$ \\
  $1$ &$4$&$(1,2,6,10,17,20,24,20,17,10,6,2,1)$ \\
  $2$ &$4$&$(1,3,10,22,41,60,73,73,60,41,22,10,3,1)$ \\
  \hline
\end{tabular}
\end{equation*}

\section*{Appendix: $q$-Binomial coefficients}\label{append}
 The \emph{q-binomial coefficients} are classical combinatorial expressions with
 diverse interpretations. They arise in number theory, combinatorics,
 linear algebra, and finite geometry. \\

We define for $0\leq k \leq n$ the \emph{q-binomial coefficient} ${n\choose
  k}_q$, first introduced by Gauss, as
$${n\choose
  k}_q=\frac{(1-q^n)(1-q^{n-1})\cdots(1-q^{n-k+1})}{(1-q^k)\cdots(1-q)}.$$

The following basic identities are proven in \cite{Gau63}:
\begin{eqnarray*}
\label{id1}{n \choose k}_q&=&{n \choose n-k}_q,\\
\label{id2}{n \choose k}_q&=&{n-1 \choose k}_q+q^{n-k}{n-1 \choose k-1}_q,\\
\label{id3}{n \choose k}_q&=&q^k{n-1 \choose k}_q+{n-1 \choose
  k-1}_q,\\
\label{id4}\sum_{j=0}^a q^j\binom{d+j}{j}_q&=&\binom{d+a+1}{a}_q.
\end{eqnarray*}

One further identity, established using the basic identities in
$q$-binomial coefficients above, is%\eqref{id1} -- \eqref{id4}, is
\begin{eqnarray*}
\label{id5}\sum_{i+j=a}q^{(c+1)j}{c+i \choose i}_q{d+j \choose
  j}_q&=&{c+d+a+1\choose a}_q.
\end{eqnarray*}
More general identities can be obtained for
$\sum_{i+j=a}q^{(c+n)j}{c+i \choose i}_q{d+j \choose j}_q$, whenever
$n>0$, using the previous identities. Sums of this kind
are used in the proof of Theorems \ref{gkn2} and \ref{gkn3}, but they will not be
stated explicitly.

\subsubsection*{Acknowledgments} The author is indebted to Andrew Kresch for his advice and constant support. He also thanks Pedro Marques and Loring Tu for many helpful discussions. 

\end{document}